\documentclass[leqno]{amsart}

\usepackage{amscd,amssymb,amsmath}

\usepackage[bookmarks]{hyperref}
\hypersetup{backref, colorlinks=true}
\usepackage{tikz}

 \usepackage{graphicx}
\usepackage[all]{xy}
\newtheorem{theorem}{Theorem}[section]
\newtheorem{lemma}[theorem]{Lemma}

\newtheorem{corollary}[theorem]{Corollary}
\newtheorem{remark}[theorem]{Remark}

\newtheorem{example}[theorem]{Example}
\newtheorem{definition}[theorem]{Definition}

\newenvironment{equationth}{\stepcounter{theorem}\begin{equation}}{\end{equation}}

\def\rond{\mathaccent"7017}

\begin{document}

\title[]{A singular variety associated to the smallest degree Pinchuk map} 
\makeatother

\author[Nguyen Thi Bich Thuy]{Nguyen Thi Bich Thuy}
\address[Nguyen Thi Bich Thuy]{UNESP, Universidade Estadual Paulista, ``J\'ulio de Mesquita Filho'', S\~ao Jos\'e do Rio Preto, Brasil}
\email{bichthuy@ibilce.unesp.br}
\maketitle \thispagestyle{empty}

\begin{abstract}
n \cite{Valette}, the authors associated to a nonvanishing Jacobian polynomial map $F : \mathbb{C}^2 \to \mathbb{C}^2$ singular varieties whose intersection homology  describes the geometry of singularities of the map. 
We describe such a variety associated to the smallest degree Pinchuk map and we calculate its intersection homology.  \end{abstract}

\section{Introduction} 
Let $F: \mathbb{K}^n\to\mathbb{K}^n$, where $\mathbb{K} = \mathbb{C}$ or $\mathbb{K} = \mathbb{R}$, be a polynomial map and let us denote by $JF(x)$ the Jacobian matrix of $F$ at $x \in \mathbb{K}^n$.  The determinant $\det (JF(x))$ is a polynomial map from $\mathbb{K}^n$ to $\mathbb{K}$. 
In 1939, O. H. Keller \cite{kel} stated a famous conjecture known nowadays as the  Jacobian conjecture, whose statement is the following:

``A polynomial map $F: \mathbb{K}^n\to\mathbb{K}^n$ is nowhere vanishing Jacobian, {\it i.e.} ${\rm det} (JF(x)) \neq 0$, for any $x \in \mathbb{K}^n$, if and only if it is a polynomial automorphism.'' 

If $F$ has a global inverse, then its inverse, being continous, maps compact sets into compact sets, in other words $F$ is proper. 
The smallest set $S_F$ such that the map $F:\mathbb{K}^n \setminus F^{-1}(S_F)\to \mathbb{K}^n\setminus S_F $ is proper 
is called the asymptotic set of $F$. 
The Jacobian conjecture  reduces to show that the asymptotic set of  a polynomial  map satisfying the  nonvanishing Jacobian condition is empty.

In the complex case, the Jacobian conjecture remains open today even for the dimension 2. However, in the real case, the 2-dimensional Jacobian conjecture was solved negatively  by Pinchuk \cite{Pinchuk}  in the year 1994. 
 In fact, Pinchuk provided counter-examples by giving a class of polynomial maps 
 $(p,q) : \mathbb{R}^2 \to \mathbb{R}^2$  satisfying  the condition ${\rm det} (J(p,q)(x,y)) >0$ for every 
 $(x,y) \in \mathbb{R}^2$ but $F$ is not injective.  Let us recall his construction: given $(x,y)\in \mathbb{R}^2$, denote
$$t= xy -1, \quad h=t(xt+1), \quad f=(xt+1)^2(t^2+y).$$
Then the Pinchuk maps $(p, q)$ are the ones with $p= f+h$ and $q$ varies 
 for different maps $(p,q)$ but $q$ always has the form 
$$q=-t^2 -6th(h+1) -u(f,h)$$
where $u$ is an auxiliary polynomial in $f$ and $h$ and is chosen such that 
$${\rm det} (J(p,q)) = t^2 + (t+ (f(13+15h))^2 +f^2$$
(\cite[Lemma 2.2]{Pinchuk}). 
Then ${\rm det} (J(p,q)(x,y)) >0$, for every $(x,y) \in \mathbb{R}^2$ since if $t = 0$ then $f= y = \frac{1}{x} \neq 0$.  

As the Remark at the end of the paper  \cite{Pinchuk}, 
the Pinchuk map constructed in the proof of  \cite[Lemma 2.2]{Pinchuk} has degree 40, 
where ${\rm deg}(p) = 10$ and  ${\rm deg}(q) = 40$, 
but one can reduce ${\rm deg} (q)$ to 35 by a suitable choice of the 
auxiliary polynomial $u(f,h)$. 
In \cite{Essen} (page 241), Arno van den Essen  choose
$$u(f, h) = 170fh + 91h^2 + 195fh^2 + 69h^3 + 75fh^3 + \dfrac{75}{4}h^4,$$
and in this case the degree of the Pinchuk map is 25. 
This is also the one that Andrew Campbell studies in the series of his papers 
 \cite{Andrew1, Andrew2, Andrew3, Andrew4, Andrew5}. 
 Let us denote by $\mathcal{P}$ this Pinchuk map. 
 
  This paper is inspired by the paper \cite{Valette}: in the year 2010, Anna Valette and Guillaume Valette gave  a new approach to study the complex Jacobian conjecture. 
Given a polynomial map $F: \mathbb{C}^n \to \mathbb{C}^n$, 
 they constructed some real $2n$-dimensional pseudomanifolds $V_F$ contained in some Euclidean space $\mathbb{R}^{\nu}$, where $\nu > 2n$ and such that 
the singular loci of these pseudomanifolds are contained in $(S_F \times K_0(F)) \times \{ 0_{\mathbb{R}^{\nu -2n}}\}$,  where $K_0(F)$ is the set of critical values of $F$. 
 In the case of dimension 2, they prove that for a non\-va\-nishing Jacobian polynomial map $F: \mathbb{C}^2 \to \mathbb{C}^2$ ({\it i.e}, $K_0(F) = \emptyset$), the condition ``$S_F = \emptyset$'' is equivalent to the condition ``the intersection homology in dimension two and with any perversity of any constructed pseudomanifold $V_F$ is trivial'' (Theorem 3.2 in \cite{Valette}). This result is generalized 
 in the case of higher dimension in \cite{ThuyValette}. 
 Moreover, the varieties defined by Anna and Guillaume Valette can also be modified to study the bifurcation set of a polynomial map $G: \mathbb{C}^n \to \mathbb{C}^{n-1}$ (\cite{ThuyCidinha}).

 We call {\it Valette varieties} for singular varieties  $V_F$ constructed by Anna and Guillaume Valette in  \cite{Valette}. Let us remark that the Valette varieties can be defined also for real polynomial maps $F: \mathbb{R}^n \to \mathbb{R}^n$ (see Remark 2.7 of \cite{Valette} or Proposition 3.8 of \cite{ThuyValette}). In this case, the Valette varieties are not necessarily pseudomanifolds, they are just semi-algebraic stratified sets. In this paper, we investigate the following two natural questions:

1) How are the behaviours of Valette varieties associated to  the Pinchuk map $\mathcal{P}$ of degree 25 mentioned above? 

2) Is there a ``real version'' of Anna and Guillaume Valette's result, {\it i.e} if $F: \mathbb{R}^2 \to \mathbb{R}^2$ is a nowhere vanishing Jacobian polynomial map then the condition $S_F = \emptyset$ is equivalent to the condition $IH_1^{\bar{0}}(V_F) = 0$? (Notice that in this case, the dimension of $V_F$ is 2, then there is only one perversity: the zero perversity).

Remark that since the Pinchuk map $\mathcal P$ satisfies the novanishing Jacobian condition then the asymptotic variety $S_{\mathcal P}$ is non-empty. We describe in this paper a Valette variety $V_{\mathcal{P}}$ associated to the Pinchuk map $\mathcal{P}$ and we calculate its intersection homology. 
The main result is Theorem \ref{Inter_Hom_Pinchuk}: the intersection homology of $V_{\mathcal{P}}$ in dimension one and with the zero perversity is trivial. 
  The main tool to prove this result is the 
  description of the behaviours of the asymptotic variety and the ``asymptotic flower'' (the inverse image of the asymptotic variety) of the Pinchuk map $\mathcal{P}$ in the papers \cite{Andrew1, Andrew2, Andrew3}. 
  
  The structrue of the paper is the following: 
   Section \ref{Intersection homology} is the preliminary about intersection homology. 
    Section \ref{ValetteVariety} gives some preliminaries about the asymptotic variety and Valette varieties of a polynomial map 
    $F: \mathbb{K}^n \to \mathbb{K}^n$ for both cases $\mathbb{K} = \mathbb{R}$ and $\mathbb{K} = \mathbb{C}$.
    We provide in this section two examples: 
    Example  \ref{example_asymp_set} is to illustrate the difference of the asymptotic set for the same map but in complex and real situations; 
    Example  \ref{example_ValetteVariety} is for illustrating the construction of Valette varieties. 
    We end the paper with Section  \ref{MainResult} where the main result 
    (Theorem \ref{Inter_Hom_Pinchuk}) is proved. 
    This result shows that there is a counter-example for a ``real version'' of the Valette's result.

\section{Intersection homology} \label{Intersection homology}
Given a variety $V$ in $\mathbb{R}^k$, we denote by ${\rm Reg} (V)$ and ${\rm Sing}(V)$ the regular and singular loci of the variety $V$, respectively. 
 Moreover, $\overline{V}$ will stand for the topological closure of $V$. The boundary of $V$ will be denoted by $\partial V$.

We briefly recall the definition of intersection homology. For details,
the readers can see 
\cite{GM1, GM2} or \cite{Jean Paul}.

\begin{definition} 
{\rm Let $V$ be an $m$-dimensional variety in $\mathbb{R}^k$.  A {\it locally topologically trivial stratification of $V$} is the data of a finite  filtration 
\begin{equationth} \label{dfn_stratification}
V = V_{m} \supset V_{m-1} \supset \cdots \supset V_0 \supset V_{-1} = \emptyset
\end{equationth}
 such that:   

1) for every $i$,  the set $S_i = V_i\setminus V_{i-1}$ is either empty or a topological manifold of dimension $i$; 

2) for every $x \in V_i\setminus V_{i-1}$, for all $i \ge 0$, there is an open neighborhood $U_x$  of $x$ in $V$, a stratified set $L_i$ and a homeomorphism 
 $$h:U_x \to (0,1)^i \times  cL_i,$$   such  that
 $h$ maps the strata of $U_x$ 
(induced stratification) onto the strata of  $  (0,1)^i \times cL_i$  (product stratification).

A connected component of $S_i$ is  called   {\it a stratum} of $V$.
}
\end{definition}

\begin{definition}
{\rm 
A {\it perversity} is an 
$(m+1)$-uple  of integers 
$$\bar p = (p_0, p_1, p_2,
p_3,\dots , p_m)$$ 
such that $p_0 = p_1 = p_2 = 0$ and $p_{r+1}\in\{p_r, p_r + 1\}$ for $ 2 \leq r \leq m-1$.

The perversity 
$\overline{0}=(0,\dots,0)$ is called the zero perversity.

In this paper, we consider  the groups of $i$-dimensional $PL$ chains $C_i(V)$ and we denote by  $c$  the support of a chain $c$.   
  A chain $c$ is $(\bar
p, i)$-{\it allowable} if  
$\dim (c \cap V_{m-r}) \leq i - r + p_r$, for all $r \geq 0$. 
It is easy to see that this condition holds always when $r =0$. 
Define $IC_i ^{\overline{p}}(V)$ to be the $\mathbb{R}$-vector subspace of $C_i(V)$
consisting of the chains $c$ such that $c$ is
$(\overline{p}, i)$-allowable and its boundary $\partial c$ is
$(\overline{p}, i - 1)$-allowable, that means 
\begin{equationth} \label{conditionallowable}
IC^{\overline{p}}_i (V) = \left\{c \in C_i(V) :  
\begin{matrix}
 & \dim ( c \cap V_{m- r}) \leq i- r + p_{r}  & \cr 
 & \dim ( (\partial c) \cap V_{m- r} \leq (i - 1) - r + p_r&  
\end{matrix}, \forall r \geq 1 \right\}.
\end{equationth}
}
\end{definition}

\begin{definition} 
{\rm The {\it $i^{th}$ intersection homology group with perversity $\overline{p}$}, denoted by
$IH_i ^{\overline{p}}(V)$, is the $i^{th}$ homology group of the
chain complex $IC^{\overline{p}}_*(V).$
}
\end{definition}
 
Recall that a {\it pseudomanifold} $V$ is a variety such that its singular locus 
is of codimension at least 2 in $V$ and its regular locus is dense in $V$. 
Goresky and MacPherson \cite{GM1, GM2} proved that the intersection homology of a pseudomanifold does not depend on a choice of a locally topologically trivial stratification (see also \cite{Jean Paul}).

In this paper, we consider the intersection homology with real coefficients, {\it i.e.} the intersection homology groups $IH_i^{\bar{p}}(V, \mathbb{R})$.  
Moreover,  we consider the groups of $PL$ chains with both compact supports and  closed supports.  
 Given a triangulation of $V$, 
recall that a chain with compact support is a chain of the form $\sum c_\sigma \sigma$  
for which all coefficients $c_\sigma \in \mathbb{Z}$ 
are zero but a finite number, where $\sigma$ are $i$-simplices. 
 A chain with closed support is a locally finite linear combination 
$\sum c_\sigma \sigma$ 
with integer coefficients $c_\sigma \in \mathbb{Z}$. 
Notice that if $c$ is a chain with compact support, then $c$ is also a chain with closed support.

The homology groups with closed supports are called Borel Moore homology, or homology groups of {\it  deuxi\`eme esp\`ece} in \cite {Cartan}. 
The intersection homology groups with closed supports are  denoted by 
$IH_i^{\bar{p},{\rm cl}} (V)$. 
  The corresponding intersection homology groups  with compact supports are denoted by 
  $IH_i^{\bar{p},\rm c} (V)$.

\section{Asymptotic variety and Valette varieties}  \label{ValetteVariety}

\subsection{Asymptotic variety}

Let $F : \mathbb{K}^n \to \mathbb{K}^n$ be a polynomial map, where $\mathbb{K}= \mathbb{C}$ or $\mathbb{K} = \mathbb{R}$. We denote  by $S_F$ the set of points at which the map $F$ is not proper, {\it i.e.} 
$$S_F = \{ \alpha \in \mathbb{K}^n \text{ such that } \exists \{ \xi_k\} \subset \mathbb{K}^n, \vert \xi_k \vert \to \infty, F(\xi_k) \to \alpha \},$$
and call it the {\it asymptotic variety}. Notice that, by $\vert \xi_k \vert$  
we mean the usual Euclidean norm of $\xi_k$ in $\mathbb{K}^n$. 
In the complex case, one has:

\begin{theorem} [\cite{Jelonek1}] \label{theo-Jelonek1} 
{\rm If $F: \mathbb{C}^n \to \mathbb{C}^n$ is a generically finite polynomial map, then $S_F$ is either an $(n-1)$ pure dimensional  algebraic variety or the empty set.}
\end{theorem}

Recall that one says that $F$ is generically finite if there exists a subset  $U \subset \mathbb{C}^n$ dense  in  the target space $\mathbb{C}^n$ such that  for any 
$a \in U$, the fiber $F^{-1} (a)$ is finite. 

In the real case, if the asymptotic variety is non-empty, then its dimension can be any integer between 1 and $(n-1)$:

\begin{theorem} [\cite{Jelonek2}] \label{theo-Jelonek2}
{\rm Let $F: \mathbb{R}^n \to \mathbb{R}^n$ be a non-constant polynomial map. 
Then the set $S_F$ is  a closed, semi-algebraic set and for every non-empty connected component $S \subset S_F$, we have $1 \leq \dim S \leq n-1$. 
}
\end{theorem}

In the following we provide an example to illustrate Theorem \ref{theo-Jelonek1} and  Theorem \ref{theo-Jelonek2} and to show that the asymptotic set of the same map may be different in complex and real situations. 

\begin{example} \label{example_asymp_set}
{\rm Let $F: \mathbb{C}^3 \to \mathbb{C}^3$ such that 
$$F(x, y,z) = (x, y, (x^2 + y^2) z).$$
It is easy to see that $F$ is generically finite since the subset $U=\{(\alpha, \beta, \gamma) \in \mathbb{C}^3: \alpha^2 + \beta^2 \neq 0\}$  is dense in  the target space and $F\vert_{F^{-1}(U)}: {F^{-1}(U)} \to U$ is bijective. 

We determine now the asymptotic variety of $F$. Assume that $\{\xi_k\}=\{(x_k, y_k, z_k)\}$
 is a sequence  
tending to infinity  in the source space such that its image does not tend to infinity. 
Hence the coordinates $x_k$ and $y_k$ cannot tend to infinity. 
Therefore, $z_k$ must tend to infinity. 
Since $(x_k^2 + y_k^2) z_k^2$ cannot tend to infinity, 
then $(x_k^2 + y_k^2)$ must tend to zero. 
Consequenty, the asymptotic variety of $F$ is the algebraic variety of equation $\alpha^2 + \beta^2 = 0.$ 

As an illustration, let us take the sequence $\{ (a + 1/k, b + 1/k, c k) \}$  tending to infinity in the source space such that $a^2 + b^2=0$. Then its image tends to $(a, b, 2(a + b)c)$. This point belongs to the hypersurface $\alpha^2 + \beta^2 = 0.$ 

Notice that, if we replace $\mathbb{C}$ by $\mathbb{R}$, we get the same equation for the asymptotic variety. However, the equation $\alpha^2 + \beta^2 = 0$ reduces to a line  in $\mathbb{R}^3$, which is not  anymore a hypersurface.  

}
\end{example}

\subsection{Valette varieties}

\subsubsection{Construction} \label{ConstructionValette}

Valette varieties $V_F$ are constructed originally in \cite{Valette}.  
In this section, we recall briefly  this construction \cite[Proposition 2.3]{Valette}: 
Let $F: \mathbb{C}^n \to \mathbb{C}^n$ be a polynomial map, the construction of Valette varieties associated to $F$ consists of the following steps:

{\bf Step 1:} Consider $F$ as a real map $F: \mathbb{R}^{2n} \to \mathbb{R}^{2n}$. 
 Determine  the set of critical points ${\rm Sing} (F)$ of $F$.

{\bf Step 2:} Choose a covering $\{ U_1, \ldots , U_p \}$ of $M_F = \mathbb{R}^{2n} \setminus {\rm Sing}(F)$ by semi-algebraic open subsets (in $\mathbb{R}^{2n}$) such that on every element of this covering, the map $F$ induces a diffeomorphism onto its image. 
 Choose semi-algebraic closed  subsets $V_i \subset U_i$ (in $M_F$) which cover $M_F$ as well. 

{\bf Step 3:} For each $ i =1, \ldots , p$, choose a Nash function $\psi_i : M_F \to \mathbb{R}$,  such that:
\begin{enumerate}
\item[(a)] $\psi_i$ is positive on $V_i$ and negative on $M_F \setminus U_i$, 
\item[(b)] $\psi_i (\xi_k)$ tends to  zero whenever $\{\xi_k\} \subset M_F$ tends to infinity or to a point in 
${\rm Sing} (F)$.
\end{enumerate}

Recall that a  {\it Nash function}  $\psi : W \rightarrow \mathbb{R}$ defined on $W \subset \mathbb{R}^{2n}$ 
is an analytic function such that there exists a non trivial polynomial $P: W \times \mathbb{R} \to \mathbb{R}$ such that $P(x, \psi(x))=0$ for any $x \in W$. 

The existence of Nash functions $\psi_i$ comes from Mostowski's Separation Lemma (see Separation Lemma in \cite{Mos}, page 246).

{\bf Step 4:} Determine the closure of the image of $M_F$ by $(F, \psi_1, \ldots, \psi_p)$ in $\mathbb{R}^{2n+p}$, that means: 
$$V_F : = \overline{(F, \psi_1, \ldots, \psi_p)(M_F)},$$ 
we obtain a Valette variety associated to the given map $F$ and the chosen covering $\{U_1, \dots, U_p\}$ and the chosen  Nash functions $\psi_1, \dots, \psi_p$.

Notice that we may have many ways to choose the covering $\{ U_1, \ldots , U_p \}$ and, moreover, with each covering $\{ U_1, \ldots , U_p \}$, 
we may have many choices of Nash functions $\psi_1, \ldots, \psi_p$. Then we may have  more than one Valette variety associated to a given polynomial map. However, the singular locus of any Valette variety is always contained in $(K_0(F) \cup S_F ) \times \{0_{\mathbb{R}^p}\}$, which depends only on the given map $F$.

In the real case,  {\it i.e.} if $F: \mathbb{R}^n \to \mathbb{R}^n$, the real map in the first step is replaced by the map $F$ itseft and the construction follows the same process. 
However, in this case, Valette varieties are no longer  pseudomanifolds in general. They are simply real semi-algebraic varieties (see Remark 2.7 of \cite{Valette} or Proposition 3.8 of \cite{ThuyValette}).  The reason is that, in the real case, the dimension of the asymptotic variety may be greater than $n-2$ (Theorem \ref{theo-Jelonek2}, 
see Example \ref{example_asymp_set} for the illustration). 
Then $K_0(F) \cup S_F$ has no reason to be of codimension 2. 

In the following we illustrate the Valette's construction above by a simple example $F: \mathbb{R} \to \mathbb{R}$.  
For more examples, one can see  \cite[Exemple 2.8]{Valette} and  \cite[Exemple 3.9]{ThuyValette}.

\subsubsection{An example}

\begin{example}  \label{example_ValetteVariety}
{\rm 
Let $F: \mathbb{R}_x \to \mathbb{R}_\alpha$ be a polynomial map defined by $F(x) = x^2$. 
 Here by $\mathbb{R}_x$ and $\mathbb{R}_\alpha$ we denote the source space and the target space  respectively. 
 We follow the four steps in the Valette's construction in  \ref{ConstructionValette}: 
 
 \medskip 
 
 {\bf Step 1:} By an easy calculation, we have 
$$S_F = \emptyset, \quad {\rm Sing}(F) = \{0\}, \quad K_0(F) = \{ 0\}. $$

\medskip 

{\bf Step 2:} The set ${\rm Sing}(F)$ divide $\mathbb{R}_x$ into two subsets: 
$$U_1 = \{x \in \mathbb{R}_x: x >0\} \quad,  \quad U_2 = \{x \in \mathbb{R}_x: x <0\}$$
and  $F(U_1) = F(U_2) = \{ \alpha \in \mathbb{R}_\alpha: \alpha >0\}$. 

These subsets $U_1$ and $U_2$ are semi-algebraic open subsets in $\mathbb{R}_x$ and on each $U_i$, for $i=1, 2$, the map $F$ induces a diffeomorphism onto its image. 
In this case, we can choose $\{U_1, U_2\}$  as a covering  of $M_F = \mathbb{R} \setminus {\rm Sing} (F) = \mathbb{R} \setminus \{0\}$. 

We choose $V_1 = U_1$ and $V_2 = U_2$. In this case, $V_1$ and $V_2$ are closed subsets of $M_F$ and they cover $M_F$ as well. 

\medskip 

{\bf Step 3:} We choose the Nash functions: 
$$\psi_1: M_F \to \mathbb{R}, \quad \psi_1(x):= \frac{x}{1+x^2} 
\; \; \; \text{ and } \; \; \;  \psi_2: M_F \to \mathbb{R}, \quad \psi_2(x):= \frac{-x}{1+x^2}.$$ 
We see that: 
\begin{enumerate}
\item $\psi_1$ is a Nash function because there exists the polynomial $P(x,y) = (x^2+1)y - 1$ such that $P(x, \psi(x)) =0$ for any $x \in M_F$. 
With the same way, $\psi_2$ is also a Nash function. 
\item $\psi_i$ is positive on $V_i = U_i$ and negative on $M_F \setminus U_i = U_j$, for $i =1, 2$ and $j \neq i$. 
\item If $\{x_k\} \subset M_F$ is a sequence tending to infinity, then $\psi_i(x_k)$ tends to $0$ and, 
if $\{x_k\} \subset M_F$ is a sequence tending to $0 \in {\rm Sing}(F)$, then $\psi_i(x_k)$ also tends to $0$. 
\end{enumerate}
Then the functions $\psi_i$ satisfy all the properties in the Valette's construction. 

\medskip

{\bf Step 4:} Now, in order to determine the Valette variety associated to the given $F$ and  the chosen covering $\{U_1, U_2\}$ and the chosen Nash functions $\psi_1, \psi_2$, we need to calculate the closure of $(F, \psi_1, \psi_2)(M_F)$ in $\mathbb{R}^3$, that means, we have to calculate   
$$
\overline{(F, \psi_1, \psi_2)(M_F)} =  \overline{ \left\{ \left( x^2, \frac{x}{1+x^2}, \frac{-x}{1+x^2} \right) : x \in M_F\right\}}.$$
By item (3) in Step 3, we have: 
  
  $$V_F= \left\{ \left( x^2, \frac{x}{1+x^2}, \frac{-x}{1+x^2} \right) : x \neq 0 \right\} \cup \{ (\alpha,0,0) : \exists \{ x_k\} \subset M_F$$  
  $$ \text{ such that } x_k \to 0 \text{ and } F(x_k) \to \alpha\}.$$

It is easy to see that if a sequence $\{ x_k\}$ tends to zero, then $(F, \psi_1, \psi_2)(x_k)$ tends to the origin of $\mathbb{R}^3$, 
which coincides with $(F, \psi_1, \psi_2)(0)$. 
It follows that the closure of  $(F, \psi_1, \psi_2)(M_F)$ is, in fact, the set $(F, \psi_1, \psi_2)(\mathbb{R}^2)$, which is smooth. Then the Valette variety $V_F$ associated to $F(x) = x^2$ and the chosen covering $\{U_1, U_2\}$ and the chosen Nash functions $\psi_1, \psi_2$ above is a smooth space  curve and it can be  graphed  
\footnote{The figure is graphed by the Software online: 
\url{http://www.math.uri.edu/~bkaskosz/flashmo/parcur/}}  in Figure \ref{Valette_Variety_F=x2} (the Valette variety in this case in the red curve). 

\begin{figure}[h!]  
\begin{center}
\includegraphics[scale=0.45]{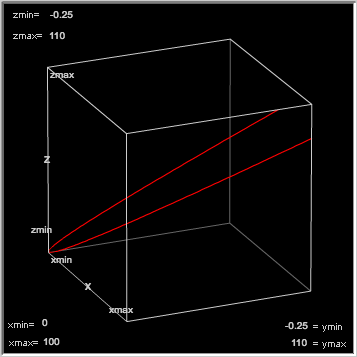} 
\caption{A Valette variety of the polynomial map $F(x) =x^2$.} \label{Valette_Variety_F=x2}
\end{center}\end{figure}

}
\end{example}

\subsubsection{Intersection homology of Valette singular varieties} 

 The intersection homology of Valette semi-algebraic pseudomanifolds $V_F$ associated to a given nonvanishing Jacobian polynomial map $F : \mathbb{C}^2 \to \mathbb{C}^2$  describes the geometry of singularities at infinity of the map (Theorem 3.2 of \cite{Valette}). This result is generalized in  \cite[Theorem 4.5]{ThuyValette} in the case of higher dimensions. In the following we simply recall these results. Notice that   
the following results hold for any Valette variety associated to a given polynomial map. 
Notice  also that the following theorems hold for  the intersection homology groups with both compact supports and closed supports.

\begin{theorem} [\cite{Valette}] \label{valette4}
{\rm Let $F: \mathbb{C}^2 \to \mathbb{C}^2$ be a polynomial map with nowhere vanishing Jacobian. The following conditions are equivalent:
\begin{enumerate}
\item $F$ is non proper.

\item  $IH_2^{\overline{p}}(V_F, \mathbb{R}) \neq 0$ for any (or some) perversity $\overline{p}$.
\end{enumerate}
}

\end{theorem}

\begin{theorem} [\cite{ThuyValette}]
{\rm Let $F = (F_1, \ldots, F_n): \mathbb{C}^n \rightarrow \mathbb{C}^n$ be a polynomial map with nowhere vanishing Jacobian. 
If {${\rm Rank}_{\mathbb{C}} {(D \hat {F_i})}_{i=1,\ldots,n} > n-2$}, where $\hat {F_i}$ is the leading form of $F_i$ and  $D \hat {F_i}$ is the (first) derivative of  $\hat {F_i}$, then the following conditions are equivalent:

\begin{enumerate}
\item $F$ is non proper.
\item $IH_2^{\overline{p}} (V_F, \mathbb{R}) \neq 0$ for any (or some) perversity $\overline{p}.$
\item $IH_{2n-2}^{\overline{p},cl} (V_F, \mathbb{R}) \neq 0$ for any (or some) perversity $\overline{p}.$
\end{enumerate}
}
\end{theorem}

\section{Intersection Homology of  a Valette variety associated to the Pinchuk map $\mathcal{P}$} \label{MainResult}

Let us recall that by the Pinchuk map $\mathcal{P}$ we mean the smallest degree one mentioned in  \cite[page 241]{Essen} and studied in the series paper \cite{Andrew1, Andrew2, Andrew3, Andrew4, Andrew5}, constructed as follows: denote
$$t= xy -1, \quad h=t(xt+1), \quad f=(xt+1)^2(t^2+y)$$
 then $\mathcal{P} = (p, q)$ with
$$p= f+h, \quad q=-t^2 -6th(h+1) -u(f,h)$$
where 
$$u(f, h) = 170fh + 91h^2 + 195fh^2 + 69h^3 + 75fh^3 + \dfrac{75}{4}h^4.$$
 In this case ${\rm det} (J\mathcal(P)(x,y)) >0$ for any $(x, y) \in \mathbb{R}^2$ but $\mathcal{P}$ is not injective. 

In order to prove the main result of this paper (Theorem \ref{Inter_Hom_Pinchuk}), 
we need Lemma \ref{LemmaVF} and Lemma \ref{Lemma-covering}. 

\begin{lemma} \label{LemmaVF}

Let $F: \mathbb{K}^n \to \mathbb{K}^n$ be a polynomial map, where $\mathbb{K} = \mathbb{C}$ or $\mathbb{C}= \mathbb{R}$. 
Then any Valette variety $V_F$ is non-compact and 
\begin{equationth} \label{Equ-LemmaVF}
V_F= (F, \psi_1, \ldots, \psi_p)(M_F) \cup ((K_0(F) \cup S_F) \times \{0_{\mathbb{R}^p}\}). 
\end{equationth}
\end{lemma}

\begin{proof}

 The fact that Valette varieties are non-compact comes directly from its definition 
 (see Section  \ref{ConstructionValette}) .  
 
 Now to prove the equality (\ref{Equ-LemmaVF}) for the complex case $\mathbb{K} = \mathbb{C}$ and the real case follows with the same way. 
 We prove at first 
$(F, \psi_1, \ldots, \psi_p)(M_F) \cup ((K_0(F) \cup S_F) \times \{0_{\mathbb{R}^p}\}) \subset V_F$. 
 We observe first that it comes directly form the definition that $(F, \psi_1, \ldots, \psi_p)(M_F) \subset V_F$. 
 Then we need to prove only that $S_F  \times {\{0_{\mathbb{R}^p}\}}$ and  $K_0(F) \times \{0_{\mathbb{R}^p}\}$ are contained in $V_F$. 

  Let us take a point $\alpha \in S_F$. Then there exists a sequence $\{\xi_k\} \subset \mathbb{C}^n$ (in the source space) tending to infinity such that $F(\xi_k)$ tends to $\alpha$. Let $\{\xi_{k_\ell}\}$ 
be a subsequence of   $\{\xi_k\}$ consisting of all points which do not belong to ${\rm Sing} (F)$, then $\{\xi_{k_\ell}\} \subset M_F$. 
Notice that the sequence $\{\xi_{k_\ell}\}$ also tends to infinity. 
Moreover, $F(\xi_{k_\ell})$  tends to $\alpha$. 
 By the construction of Valette varieties (see Section \ref{ConstructionValette}), 
 the image of $\xi_{k_\ell}$ by Nash function $\psi_i$ tend to zero, for $i=1, \ldots, p$. 
Then $(F, \psi_1, \ldots, \psi_p)(\xi_{k_\ell})$ tends to 
$(\alpha, 0_{\mathbb{R}^{p}})$, it follows that $(\alpha, 0_{\mathbb{R}^{p}})$ is an accumulation point of $(F, \psi_1, \ldots, \psi_p)(M_F)$. 
By the definition of $V_F$, the point $(\alpha, 0_{\mathbb{R}^{p}})$ belongs to $V_F$.
 Consequently, the subset $S_F  \times \{0_{\mathbb{R}^p}\}$ is contained in $V_F$. 

 We prove now that $K_0(F) \times \{0_{\mathbb{R}^p}\}$ is contained in $V_F$. For this, let us take a point $\alpha \in K_0(F)$. Then there exists a point $x \in {\rm Sing} (F)$ such that $\alpha = F(x)$. 
 Take a sequence $\{\xi_k\} \subset M_F$ such that $\xi_k$ tends to $x$. Again, by the construction of Valette varieties, the image  of $\xi_k$ by the chosen Nash function $\psi_i$  must tend to zero, for $i = 1, \ldots, p$. Then $(F, \psi_1, \ldots, \psi_p)(\xi_k)$ tends to $(F(x), 0_{\mathbb{R}^{p}}) = (\alpha, 0_{\mathbb{R}^{p}})$. That means  $(\alpha, 0_{\mathbb{R}^{p}})$, which is a point in $K_0(F) \times \{0_{\mathbb{R}^p}\}$, is 
an accumulation point of $(F, \psi_1, \ldots, \psi_p)(M_F)$. Consequently, $(\alpha, 0_{\mathbb{R}^{p}})$ belongs to $V_F$.

We prove now the inclusion $V_F \subset (F, \psi_1, \ldots, \psi_p)(M_F) \cup ((K_0(F) \cup S_F) \times \{0_{\mathbb{R}^p}\})$. 
Assume that $\beta \in V_F$ and $\beta \notin (F, \psi_1, \ldots, \psi_p)(M_F)$, we will prove that 
$\beta \in (K_0(F) \cup S_F) \times \{0_{\mathbb{R}^p}\}$. 
Since $\beta \in V_F$ and  $\beta \notin (F, \psi_1, \ldots, \psi_p)(M_F)$ 
then by the definition of $V_F$, 
there exists a sequence $\{\xi_k\} \subset M_F$ such that $(F, \psi_1, \ldots, \psi_p)(\xi_k)$ tends to $\beta$. Assume that $\xi_k$ tends to $x$, we claim that $x$ must be infinity or a point of  ${\rm Sing} (F)$. 
 In fact, if $x$ is neither infinity nor a point in  ${\rm Sing} (F)$, then $x$ belongs to $M_F$. Since $F$ is a polynomial map and $\psi_1, \ldots, \psi_p$ are Nash functions, then $(F, \psi_1, \ldots, \psi_p)(\xi_k)$ tends to  $(F, \psi_1, \ldots, \psi_p)(x)$. 
It follows $\beta= (F, \psi_1, \ldots, \psi_p)(x)$ and then $\beta$ belongs to 
$(F, \psi_1, \ldots, \psi_p)(M_F)$ and that provides a contradiction. 
Now if $\xi_k$ tends to infinity, then since $(F, \psi_1, \ldots, \psi_p)(\xi_k)$ does not tend to infinity, $F(\xi_k)$ must tend to a point in $S_F$.
If $\xi_k$ tends to a point $x \in {\rm Sing} (F)$, then $F(\xi_k)$ tends to a point in $K_0(F)$. In both of these cases, $\psi_i(\xi_k)$ tends to zero, for $i=1, \ldots, p$ and it follows $\beta \in (K_0(F) \cup S_F) \times \{0_{\mathbb{R}^p}\}.$ 
\end{proof}

\begin{lemma} \label{Lemma-covering}

There exists a covering $\{ U_1, U_2, U_3, U_4 \}$ of $\mathbb{R}^{2}$ by semi-algebraic open subsets such that on every element of this covering, the Pinchuk map $\mathcal{P}$ induces a diffeomorphism onto its image. 
Moreover, there exists also semi-algebraic closed  subsets $V_i \subset U_i$ in $\mathbb{R}^2$ which cover $\mathbb{R}^2$ as well. 
\end{lemma}

\begin{proof}

Let us remark first that since $K_0(\mathcal{P}) = \emptyset$, 
then the singular locus ${\rm Sing} (\mathcal{P})$ of $\mathcal{P}$ is empty. 
The asymptotic variety of the Pinchuk map $\mathcal{P}$ and its inverse image, which is called 
{\it asymptotic flower} in the papers \cite{Andrew1, Andrew2, Andrew3}, are fully described and sketched in these papers.  
In Figure \ref{DessinCampell}, we copy from Figure 2 and Figure 3 from \cite[ Section 5, pages 31-33]{Andrew3} \footnote{The use of these figures was asked the authorization of the author.}  of the asymptotic set and the asymptotic flower of the Pinchuk map $\mathbb{P}$.

\begin{figure}[h!]  
\begin{center}
\includegraphics[scale=0.8]{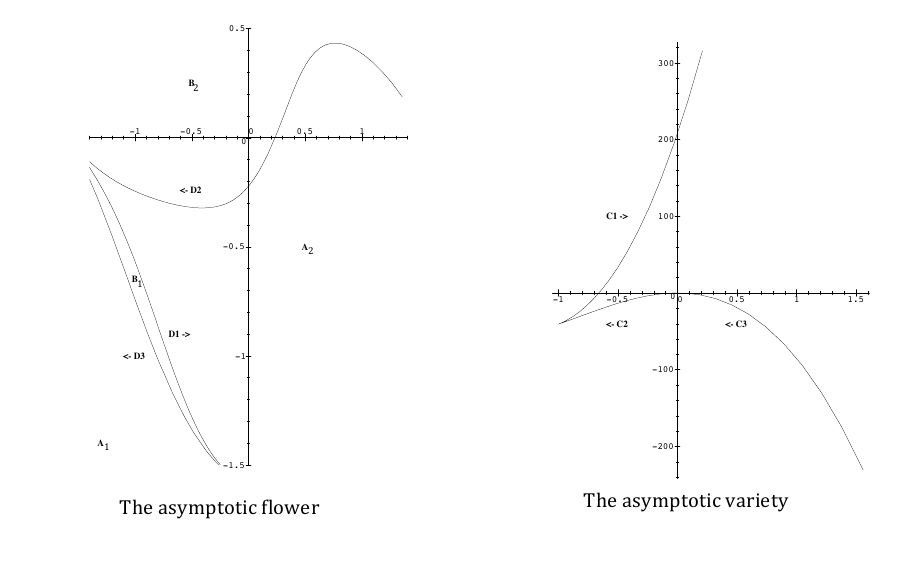} 
\caption{The ``asymptotic flower''  and the asymptotic variety of $\mathcal{P}$ described in \cite{Andrew3}.}
\label{DessinCampell} 
\end{center}\end{figure}

We use the following properties: 

\begin{enumerate}

\item (See \cite{Andrew1}, page 1) The asymptotic variety $S_{\mathcal{P}}$ of the Pinchuk's map $\mathcal{P}$  is a curve parametrized by the bijective polynomial: 
$$p(s) = s^2-1, \quad  \quad q(s) = -75 s^5 + \frac{345}{4}s^4 - 29s^3 + \frac{117}{2}s^2-\frac{163}{4},$$
where $s \in \mathbb{R}$. 

\item (See  \cite[Section 5, pages 32]{Andrew3}, see also \cite{Andrew2} and \cite[page 1]{Andrew1}) The asymptotic variety $S_{\mathcal{P}}$ intersects  
the vertical axis at $(0,0)$ and $(0,208)$ and its leftmost point is $(-1, -163/4)$. 
This leftmost point is also the only singular point of the curve $S_{\mathcal{P}}$.

\item (See  \cite[Section 5, pages 32]{Andrew3}, see also \cite{Andrew2})
 Exactly two points, the leftmost point $(-1, -163/4)$ and the origin $(0,0)$, have no inverse images. Moreover, every other point of the asymptotic variety $S_{\mathcal{P}}$ has exactly one inverse image. 
Furthermore, these two points $(0,0)$ and $(-1, -163/4)$ break up the asymptotic variety into three connected curves $C_1$, $C_2$ and $C_3$ (see Figure \ref{DessinCampell}).

\item (See  \cite[Section 5, pages 32]{Andrew3}, see also \cite{Andrew2}) The asymptotic flower is the union of three curves $D_1$, $D_2$ and $D_3$, 
which are the inverse images of $C_1, C_2$ and $C_3$, respectively, in the following way: two points $(-1, -163/4)$ and  $(0,0)$ that have no inverse images from the asymptotic variety break up the asymptotic flower into three connected curves $D_1, D_2$ and $D_3$ such that each point of each of the three curves $C_1, C_2$ and $C_3$ has exactly one inverse image (see Figure \ref{DessinCampell}).

\item (See  \cite[Section 5, pages 33]{Andrew3}, see also \cite{Andrew2}) Each of the three curves $D_1$, $D_2$, $D_3$ divides  the source space into simply connected parts, which may be described as the regions left and right of the curve, using the induced orientations to define left and right. Removing the curves $D_1, D_2$ and $D_3$, it leaves exactly four simply connected open components: two regions $A$ and two regions $B$.  
In Figure \ref{DessinCampell}, we denote by $A_1$ ({\it resp.}, $B_1$) the region $A$ ({\it resp.}, $B$) on the left and $A_2$ ({\it resp.}, $B_2$) is the region $A$ ({\it resp.}, $B$) on the right.
 Remark that the restriction of $\mathcal{P}$ to each of the regions $A_1, A_2, B_1, B_2$ is homeomorphism (see Figure \ref{DessinCampell}). 
\end{enumerate}

By the above properties, one can choose the semi-algebraic closed subset $\{V_i\}_{i=1,\ldots, \,4}$ as follows: 
$$V_1= A_1 \cup D_3, \quad V_2= B_1 \cup D_3 \cup D_1, \quad V_3 = A_2 \cup D_1 \cup D_2, \quad V_4= B_2 \cup D_2.$$

Moreover, by the item (5) above, one can choose an covering $\{U_i\}_{i=1,\ldots,\,4}$ of $\mathbb{R}^2$ by semi-algebraic open subsets such that: 
$U_i \supset V_i$ and the restriction  $\mathcal{P} \vert_{U_i}$ is  homeomorphism, for $i= 1, \ldots, \, 4.$ 

\end{proof}

\begin{theorem} \label{Inter_Hom_Pinchuk}  
 There exists a Valette variety $V_{\mathcal{P}}$ associated to the Pinchuk map $\mathcal{P}$  such that
$IH_1^{\bar{0}, {\rm c}} (V_\mathcal{P})=  IH_1^{\bar{0}, {\rm cl}} (V_\mathcal{P})= 0$.
\end{theorem}

\begin{proof} 
By the construction of Valette varieties (Section \ref{ConstructionValette}) and Lemma \ref{LemmaVF} 
and Lemma \ref{Lemma-covering}, the Valetty varieties  $V_\mathcal{P}$ of the Pinchuk  map $\mathcal{P}$ associated the covering $\{U_i\}_{i=1,\ldots,4}$ chosen in Lemma  \ref{Lemma-covering} has four components of dimension 2: they are smooth, non-compact, connected and are glued along the asymptotic variety $S_{\mathcal{P}}$ as follows: 
$F(A_2)$ and $F(B_1)$ are glued along the curve $C_1$, 
$F(A_2)$ and $F(B_2)$ are glued along the curve $C_2$, 
and $F(A_1)$ and $F(B_1)$ are glued along the curve $C_3$. 
These Valette varieties are 2-dimensional varieties embedded in  $\mathbb{R}^6$  and 
 its planar presentations can be illustrated (topologically) as in Figure \ref{ValetteVarietyPinchuk}. 
 
 \begin{figure}[h!]  
\begin{center}

\begin{tikzpicture} 

\draw[dotted, thick] (3,0) parabola[bend pos=0.5] bend +(0,-1) +(-3,-2);

\draw[dotted, thick] (3,-2) parabola[bend pos=0.0] bend +(0,0) +(0,3);

\draw[dotted, thick] (6,2) parabola[bend pos=0.0] bend +(0,0) +(0,-4);

\draw[dotted, thick] (6,2) parabola[bend pos=0.0] bend +(0,0) +(-2,2);

\draw[dotted, thick] (6.2, 4) parabola[bend pos=0.] bend +(0,0) +(-0.2,-2);

\draw[dotted, thick]  (9, -2) -- (8,-2) parabola[bend pos=0.0] bend +(0,0) +(-2,4);

\fill[fill=blue!20]  (0,4) cos (3,0) parabola [bend pos=0.5] bend +(0,-1) +(-3,-2);

\fill[fill=green!20]  (3,-2) parabola[bend pos=0.0] bend +(0,0) +(0,3)  sin (6,2) parabola[bend pos=0.0] bend +(0,0) +(0,-4) ;

\fill[fill=orange!20] (0,4) cos (3,0) sin (6,2) parabola[bend pos=0.0] bend +(0,0) +(-2,2);

\fill[fill=blue!20] (6.2,4) parabola[bend pos=0.] bend +(0,0) +(-0.2, -2) cos (9,-1) -- (9,4) -- (6.2, 4);

\fill[fill=purple!20] (8,-2) parabola[bend pos=0.0] bend +(0,0) +(-2,4) cos (9,-1) -- (9, -2);

\draw[very thick] (0,4) cos (3,0);
\node at (2.1, 2)[above] {$C_1$};
\node at (3, 0)  {$\bullet$};
\node at (3.1, 0.1) [above] {$L$};
\draw[very thick] (3, 0) sin (6,2);
\node at (4.8, 1) [above] {$C_2$};
\draw[very thick] (6, 2) cos (9,-1);
\node at (7.8, 0.6) [below] {$C_3$};

\node at (3.1, 0.1) [above] {$L$};

\node at (1.5, -1)  {$F(B_1)$};
\node at (4.5, -1)  {$F(B_2)$};
\node at (8, -1)  {$F(A_1)$};
\node at (3.5, 3.3)  {$F(A_2)$};
\node at (6.9, 3.3)  {$F(B_1)$};

\draw[->] (0.5, -0.2) -- (0, -0.2);
\draw[->] (0.5, 0.2) -- (0, 0.2);
\draw[->] (0.5, 0.6) -- (0, 0.6);
\draw[->] (0.5, 1) -- (0, 1);
\draw[->] (0.5, 1.4) -- (0, 1.4);

\draw[->] (8, 2.4) -- (8.5, 2.4);
\draw[->] (8, 2.8) -- (8.5, 2.8);
\draw[->] (8, 3.2) -- (8.5, 3.2);
\draw[->] (8, 3.6) -- (8.5, 3.6);

\draw[red] (4,1) -- (2,4);
\draw[red] (4,1) -- (3.8, -2);
\node at (2.2,3)[right][red]{$\ell$};

\draw[very thick][red] (7,1.6) circle (0.3 cm); 
\node at (7.5,1.6)[red] {$c_2$};

\draw[very thick][red] (2.7, -0.5) -- (2.7, 0.6) -- (3.6, 0.6)  -- (3.1, -0.5);
\node at (3.1, 0.6) [above][red] {$c_1$};

\node at (5.9, 2) [above] {$O$};
\node at (6, 2)  {$\bullet$};
\node at (8.8, -0.5) [above] {$S_F$};



\end{tikzpicture}

\caption{The Valette variety associated to the Pinchuk map with the chosen co\-ver\-ing and some types of allowable chains for the intersection homology 
$IH_1^{\bar{0}}(V_{\mathcal{P}})$. The chains $c_1$ and $l$ are homologous for the intersection homology.} \label{ValetteVarietyPinchuk}
\end{center}
\end{figure}
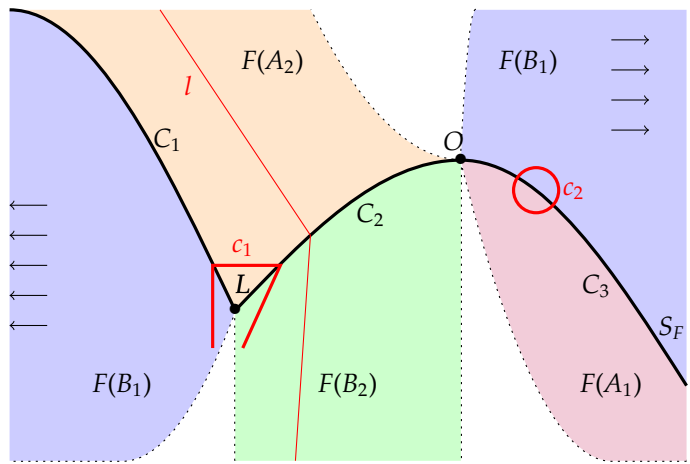
 
 Let us consider the following stratification of the varieties  $V_\mathcal{P}$: 
$$V_\mathcal{P}  \supset V_0 = \{  L, O \}  \supset \emptyset,$$
where $L= (-1, -163/4, 0_{\mathbb{R}^{4}})$ and $O=(0,0, 0_{\mathbb{R}^{4}})$. 
This stratification is locally topologically trivial stratification.  

Let us remark that since $\dim_\mathbb{R} V_\mathcal{P} = 2$, then we have only one perversity: the zero perversity $\bar{0}$ and $p_0 = p_1 = p_2 = 0$. 
We look for  the 1-dimensional allowable chains. For this, we have to verify the condition (\ref{conditionallowable}), {\it i.e.} 
for a 1-dimensional chain $ c$ be $(\bar{0},1)$-allowable, at first, $c$ must be satisfied the condition 
$$\dim (c \cap V_{2 - 2}) \leq 1 - 2 + 0.$$
Hence, if $c$ is $(\bar{0},1)$-allowable then $\dim (c \cap V_0) \leq -1$. 
Consequently, the chain $c $ cannot contain neither $L$ nor the origin $O$. 
In this case, the boun\-da\-ry $\partial c$ of $c$, which consists of  points or is an emptyset, does not meet the set $\{L, O\}$. Thus  $\partial c$ is also $(\bar{0},1)$-allowable, since $\dim  (\partial c \cap V_0) = \dim \emptyset = - \infty$. Then the 1-dimensional allowable  chains of $V_\mathcal{P} $ are the chains of types $c_1$ (with closed supports) and $c_2$ (with compact supports)  (see Figure \ref{ValetteVarietyPinchuk}). 
Notice that the chain $\ell$ as shown in Figure \ref{ValetteVarietyPinchuk} is the same type of  $c_1$, since $\ell$ and $c_1$ are homologous chains for the intersection homology with  respect to the considered stratification. 
Now assume that $\tau_i$ is a 2-dimensional chain such that the boundary of $\tau_i$ is $c_i$, for $i= 1, 2$, then the condition (\ref{conditionallowable}) is obviously true for $\tau_i$  since  
$$\dim (\tau_i \cap V_{2 - 2}) \leq {\rm dim} (V_0) = 0 = 2 - 2 +0.$$  
This shows that the two chains $\tau_1$ and $\tau_2$ are $(\bar{0},2)$-allowable.  
Then we have $IH_1^{\bar{0}, {\rm c}} (V_\mathcal{P})= IH_1^{\bar{0}, {\rm cl}} (V_\mathcal{P}) = 0$.
\end{proof}

\begin{remark}
{\rm Since the stratification used in the proof is a locally topolo\-gi\-cally  trivial stratification then the result of the  Theorem \ref{Inter_Hom_Pinchuk} is an invariant for any (another) locally topolo\-gi\-cally trivial stratification. 
}
\end{remark}

\begin{corollary}
The Valette varieties associated to the Pinchuk map $\mathcal{P}$ constructed in the proof of the Theorem \ref{Inter_Hom_Pinchuk} is a counter example for the ``real version'' of the Valette's Theorem \ref{valette4}. 
\end{corollary}

\bibliographystyle{plain}

\end{document}